\documentclass[11pt]{amsart}

\usepackage{amssymb, amsmath, graphicx, mathrsfs, enumitem, hyperref, aliascnt, xcolor}

\newcommand{\noop}[1]{}

\newcommand{\Q}{\mathbb{Q}}
\newcommand{\R}{\mathbb{R}}
\newcommand{\Z}{\mathbb{Z}}

\renewcommand{\H}{\mathbb{H}}

\newcommand{\GL}{\mathrm{GL}}

\newcommand{\PSL}{\mathrm{PSL}}

\newcommand{\Aut}{\mathrm{Aut}}
\newcommand{\Isom}{\mathrm{Isom}}
\newcommand{\rank}{\mathrm{rank}}
\newcommand{\diag}{\mathrm{diag}}

\usepackage{color}

\newtheorem{theorem}{Theorem}[section]

\newaliascnt{lemma}{theorem}
\newtheorem{lemma}[lemma]{Lemma}
\aliascntresetthe{lemma}

\newaliascnt{cor}{theorem}
\newtheorem{cor}[cor]{Corollary}
\aliascntresetthe{cor}

\newaliascnt{prop}{theorem}
\newtheorem{prop}[prop]{Proposition}
\aliascntresetthe{prop}

\newaliascnt{con}{theorem}

\aliascntresetthe{con}

\newaliascnt{defn}{theorem}

\aliascntresetthe{defn}

\theoremstyle{remark}
\newaliascnt{remark}{theorem}

\aliascntresetthe{remark}

\newaliascnt{claim}{remark}

\aliascntresetthe{claim}

\newaliascnt{question}{remark}

\aliascntresetthe{question}

\def\sek~{\S{}}

\def\equationautorefname~#1\null{(#1)\null}
\numberwithin{equation}{section}

\title{Virtually special embeddings of integral Lorentzian lattices}

\author{Michelle Chu}
\address{Department of Mathematics, Statistics, and Computer Science, University of Illinois at Chicago, Chicago, IL 60607, USA}
\email{michu@uic.edu}

\begin{document}

\begin{abstract}
The automorphism groups of integral Lorentzian lattices act by isometries on hyperbolic space with finite covolume. In the case of reflective integral lattices, the automorphism groups are commensurable to arithmetic hyperbolic reflection groups. However, for a fixed dimension, there is only finitely many reflective integral Lorentzian lattices, and these can only occur in small dimensions.
The goal of this note is to construct embeddings of low-dimensional integral Lorentzian lattices into unimodular Lorentzian lattices associated to right-angled reflection groups. As an application, we construct many discrete groups of $\Isom(\H^n)$ for small $n$ which are C-special in the sense of Haglund-Wise. 
\end{abstract}

\maketitle

\section{Introduction}
Given a finite volume polyhedron $P$ in hyperbolic space $\H^n$, let $\Gamma$ be the group generated by the reflections on the sides of $P$. If the action of $\Gamma$ tiles $\H^n$ without interiors of copies of $P$ overlapping, we say that $\Gamma$ is a \emph{hyperbolic reflection group} and its fundamental polyhedron $P$ is a \emph{hyperbolic Coxeter polyhedron}. The quotient $\H^n/\Gamma$ is a finite-volume hyperbolic orbifold.

The theory of hyperbolic reflection groups provides many examples of finite volume hyperbolic orbifolds. However, in higher dimensions, these cease to exist \cite{MR774946,MR831049,MR842588}. 
Another way to construct finite volume hyperbolic orbifolds in any dimension is as quotients of $H^n$ by the automorphism groups $\Aut(L)$ of Lorentzian lattices $L$. These automorphism groups are examples of arithmetic groups of simplest type in $\Isom(\H^n)$.
If the subgroup generated by reflections has finite index in $\Aut(L)$, we say the lattice $L$ is \emph{reflective}. Such a subgroup is an \emph{arithmetic hyperbolic reflection group}. 

In this note we construct embeddings of lattices into unimodular lattices of higher dimension. The Lorentzian unimodular lattices $\mathbf{I}_{n,1}$ are reflective for $2\leq n\leq 19$ \cite{MR0295193,MR0476640}. Furthermore, for $2\leq n\leq 8$, these are associated to reflection groups of hyperbolic right-angled polyhedra, which are geometric right-angled Coxeter groups \cite{MR2130566}. Right-angled Coxeter groups, or RACGs, are particularly interesting because they have many nice properties which are inherited by their subgroups. For example, virtually embedding hyperbolic 3-manifold groups into RACGs has determined the virtual Haken and the virtual fibering conjectures for all finite volume hyperbolic 3-manifolds as well as LERFness of their fundamental groups \cite{virtual_haken,wise_manuscript}. In the sense of \cite{haglund_wise} we say a group is \emph{C-special} if it embeds as a quasi-convex subgroup of a RACG. 

We apply the lattice embeddings together with the explicit relationship between the unimodular lattices $\mathbf{I}_{n,1}$ and RACGs given by \cite{ERT} to construct many examples of C-special hyperbolic manifold groups in dimension 3 and 4. The following theorem extends and improves the results in \cite{bianchi} (see also \cite{thesis}) and \cite[Theorem 2.6]{DMP}.

\begin{theorem}\label{th: special}
Let $\Gamma$ be an integral arithmetic group of simplest type in $\Isom(H^3)$ or $\Isom(H^4)$. Then $\Gamma_{(2)}$, the principal congruence subgroup of level 2, is compact C-special.
\end{theorem}

For fixed dimensions, the indices of the principal congruence subgroups of level 2 contained in integral arithmetic group of simplest type are uniformly bounded. \autoref{th: special} extends \cite[Theorem 1.2, Proposition 5.3]{bianchi}\cite[Theorem 3.1, Proposition 4.3, Proposition 5.3]{thesis} and also improves on the bounds for the value of $D$ found in \cite[Proposition 2.6]{DMP} by removing the dependence on the discriminant. This gives a uniform bound for $D$ which is independent of anything to get a strengthening of \cite[Theorem 2.2]{DMP}. Prior to these results, Bergeron-Haglund-Wise showed that given an arithmetic group of simplest type in $\mathrm{O}^+(n,1)$, there exist some $m$ such that the congruence subgroup of level $m$ is special \cite{BHW}, but \autoref{th: special} shows that $m=2$ is enough for the cases included in \autoref{th: special}.

As a consequence of \autoref{th: special} we can give a slight improvement to the results in \cite{bianchi}.

\begin{prop}\label{prop: bianchi} Let $d$ be a square-free positive integer and $\mathcal{O}_d$ the ring of integers in the quadratic imaginary field $\Q(\sqrt{-d})$.
The principal congruence subgroup $\PSL(2,\mathcal{O}_d)_{(2)}$ embeds in a RACG and has index
\begin{equation*}
[\PSL(2,\mathcal{O}_d): \PSL(2,\mathcal{O}_d)_{(2)}]= \begin{cases}
	48 & \text{if }d\equiv 1,2 \mod (4) \\
	60 & \text{if }d\equiv 3 \mod (8) \\
	36 & \text{if }d\equiv 7 \mod (8) .
\end{cases}
\end{equation*}
\end{prop}

In particular, since the figure eight knot group intersects $\PSL(2,\mathcal{O}_3)_{(2)}$ in a subgroup of index 10, we also get the following corollary.

\begin{cor}
The figure eight knot complement has a special cover of degree 10.
\end{cor}

This note is organized as follows:
In \autoref{sec: prelim} we give the necessary preliminary background in integral lattices, their automorphism groups, and arithmetic groups of simplest type.
In \autoref{sec: gluing}, inspired by the lattice gluings in \cite{Allcock18}, we construct embeddings of integral lattices into unimodular lattices. In \autoref{sec: special} we use these embeddings to prove \autoref{th: special}. Finally, in \autoref{sec: example} we give an explicit example.


\section*{Acknowledgments}
The author thanks Daniel Allcock for introducing her to lattice gluing and encouraging this work. This work was supported by NSF grant DMS-1803094. 



\section{Preliminaries}\label{sec: prelim}

\subsection{Integral lattices}

A \emph{lattice} $L$ is a $\Z$-module equipped with a $\Q$-valued non-degenerate symmetric bilinear form $(\cdot,\cdot)$ on the vector space $V=L\otimes\Q$, called the \emph{inner product}. $L$ is called \emph{Lorentzian} if its inner product has signature $+^n-^1$.
The \emph{norm} of a vector $v$ is its inner product with itself $(v,v)$.
If the inner product of every pair of vectors in $L$ is $\Z$-valued, $L$ is called \emph{integral}. In what follows, let $L$ be an integral lattice unless otherwise noted.

The dual of $L$ is the lattice $L^*=\{v\in L\otimes\Q: (v,L)\in\Z \}$. Notice that $L$ is integral if and only if $L\subset L^*$. We define the \emph{discriminant group} as $\Delta(L)=L^*/L$, a finite abelian group. We will refer to the minimal number of generators of $\Delta(L)$ as $\rank(\Delta(L))$.

The \emph{determinant} of $L$, or $\det L$ is the determinant of an inner product matrix $A_L$, with respect to some $\Z$-basis of $L$. It is independent of the choice of $\Z$-basis and in fact $|\det L|=|\Delta(L)|$. If $L'$ is a sublattice of $L$ of index $d$ then $\det L'=d^2\cdot\det L$.

The $\Z$-valued inner product on $L$ extends to a $\Q$-valued inner product on $L^*$ and descends to a $\Q/\Z$-valued inner product on $\Delta(L)$.

An integral lattice $L$ is called \emph{strongly-square-free}, denoted by \emph{SSF}, if the rank of $\Delta(L)$ is at most $\frac{1}{2}\dim(L)$ and every invariant factor of $\Delta(L)$ is square-free. In other words, $\Delta(L)$ is a direct product of at most $\frac{1}{2}\dim(L)$-many finite cyclic subgroups, each of square-free order.
An integral lattice is called \emph{unimodular} whenever $\Delta(L)$ is trivial. 

Lattices may also be defined more generally over totally real number fields.

\subsection{Automorphisms and arithmetic groups}
Let $k$ be a totally real number field with ring of integers $\mathcal{O}_k$ and let $f$ be a quadratic form of signature $+^n-^1$ defined over $k$ such that for every non-identity embedding $\sigma:k\hookrightarrow\R$, the form $f^\sigma$ is positive definite. Let $\mathrm{O}(f;\R)$ denote the orthogonal group that preserves $f$ and $\mathrm{O}^+(f;\R)$ its index-two subgroup which is time-orientation preserving. 
Then the group $\Gamma=\mathrm{O}^+(f;\mathcal{O}_k) =\mathrm{O}^+(f;\R)\cap\GL_{n+1}(\mathcal{O}_k)$ is a finite-covolume discrete subgroup of $\mathrm{O}^+(f;\R)$ which is identified with $\Isom(\H^n)$ via the hyperboloid model.
The field $k$ is called the \emph{field of definition} for $\Gamma$.
Any discrete subgroup of $\Isom(\H^n)$ which is commensurable to some such $\mathrm{O}^+(f;\mathcal{O}_k)$ is called \emph{arithmetic of simplest type}.


If $L$ is an integral Lorentzian lattice, then its \emph{automorphism group} is the group
\begin{align*}
\Aut(L)
& =\{ g\in\GL(V) | Lg=L \text{ and } f_L(xg,yg)=f_L(x,y) \text{ for all }x,y\in L \} \\
& = \{ g\in\GL_{n+1}(\Z) | gA_Lg^{tr}=A_L \} \\
& =\mathrm{O}(f_L;\Z).
\end{align*}
We call $\Aut(L)$ and any finite index subgroup of it an \emph{integral arithmetic group of simplest type}. 

The principal congruence subgroup of level $m$ in $\Aut(L)$ is the subgroup
\begin{align*}
\Aut(L)_{(m)}
& = \{ g\in\GL(V) | vg\equiv v \mod m \text{ for all }x\in L \}  \\
& = \{ g\in\GL_{n+1}(\Z) | gA_Lg^{tr}=A_L \text{ and } g\equiv I_{n+1} \mod m \} \\
& =\mathrm{O}(f_L;\Z)_{(m)}
\end{align*}
where $I_{n+1}$ is the $(n+1)\times(n+1)$ identity matrix.


If $n-1$ is not divisible by 8, there is, up to isomorphism over $\Z$, a unique unimodular Lorentzian lattice of signature $+^n-^1$ denoted $\mathbf{I}_{n,1}$.
Let $q_n$ be the standard Lorentzian quadratic form 
\begin{equation}\label{eq: Fn} q_n:=-x_0^2+ x_1^2 + \dots + x_n^2. 
\end{equation}
The unimodular lattice $\mathbf{I}_{n,1}$ has automorphism group $
Aut(\mathbf{I}_{n,1})=\mathrm{O}(q_n,\Z)$ and is reflective for $n\leq 19$ \cite{MR0295193,MR0476640}.

\subsection{Invariants and existence of integral lattices}
This section assumes familiarity with Conway-Sloane $p$-adic symbols \cite[Chapter 15]{ConwaySloane}.

Over the $p$-adic integers, a form $f$ associated to a $p-adic$ lattice $L_p$ can be decomposed as a direct sum
\begin{equation}
f = f_1\oplus p f_p \oplus p^2f_{p^2} \oplus\dots\oplus q f_{q} \oplus\dots
\end{equation}
where $q$ is a $p$-power and $f_q$ is a $p$-adic integral form with determinant prime to $p$. This Jordan decomposition is unique for odd primes.

For $p$ odd, the $p$-adic symbol of $f$ is the formal product of factors $q^{\epsilon_q n_q}$ with
\begin{equation*}
\epsilon_q = \left( \frac{\det f_q}{p} \right)
\text{ and } 
n_q = \dim f_q
\end{equation*}
where $\left( \frac{\alpha}{p} \right)$ denotes the Kronecker symbol.

For $p=2$, the $2$-adic symbol of $f$ is the formal product of factors $q^{\epsilon_q n_q}_{t_q}$ or $q^{\epsilon_q n_q}$ where the former indicates $f_q$ is of type I and the later indicates $f_q$ is of type II and with
\begin{equation*}
\epsilon_q = \left( \frac{\det f_q}{2} \right)
\text{ , } 
n_q = \dim f_q
\text{ and } 
t_q = \mathrm{oddity}(f_q)
\end{equation*}
where the Kronecker symbol $\left( \frac{a}{2} \right)$ is $+1$ if $a\equiv\pm1\mod 8$ or $-1$ if $a\equiv\pm3\mod 8$.

Unfortunately, the 2-adic symbol is not unique, since a $2$-adic form can have essentially different Jordan decompositions. However, Conway-Sloane define an abreviated $2$-adic symbol using compartments and trains. Two abreviated $2$-adic symbols represent the same form if and only they are related by sign walking see \cite[Chapter 15, \S7.5]{ConwaySloane}.

By \cite[Theorem 11, Chapter 15]{ConwaySloane}, there exist an integral lattice $L$ of determinant $d$ having specified local forms $L_p$ and signature $+^r-^s$ if and only if the determinant condition, the oddity formula, and the Jordan blocks conditions displayed below hold.

\begin{itemize}
\item[1.] The determinant condition: for each $p$, the $\epsilon_q$ from the $p$-adic symbol satisfy
	\begin{equation} \label{eq: det condition}
	\prod \epsilon_q = \left(\frac{a}{p}\right)
	\end{equation}
	where $\det(L)=p^\alpha a$.
\item[2.] The oddity formula:
	\begin{equation} \label{eq: oddity formula}
	\mathrm{signature}(L)+\sum_{p \text{ odd}} p\mathrm{-excess}(L_p) \equiv \mathrm{oddity}(K_2) \mod8
	\end{equation}
	where $$\mathrm{signature}(L)= r-s,$$
	$$p\mathrm{-excess}(L_p)\equiv\sum_q n_q(q-1)+4\cdot\#(\text{odd powers $q$ with $\epsilon_q=-1$}),$$ 
	$$\text{and } \mathrm{oddity}(L_2)=\sum t_q+4\cdot\#(\text{odd powers $q$ with $\epsilon_q=-1$}) . $$
\item[3.] The Jordan blocks conditions:
the $2$-adic Jordan blocks satisfy the following
	\begin{equation} \label{eq: Jc cond 1}
 	\text{if type II, } t_q\equiv 0\mod 8
 	\end{equation}
 	\begin{equation} \label{eq: Jc cond 2}
 	\text{if }n_q=1, 
 	\begin{cases}
 	\epsilon_q=+1 \text{ then } t_q\equiv\pm 1\mod 8\\
 	\epsilon_q=-1 \text{ then } t_q\equiv\pm 3\mod 8
 	\end{cases}
 	\end{equation}
 	\begin{equation} \label{eq: Jc cond 3}
 	\text{if type I and }n_q=2,
 	\begin{cases}
 	\epsilon_q=+1 \text{ then } t_q\equiv 0 \text{ or }\pm2\mod 8\\
 	\epsilon_q=-1 \text{ then } t_q\equiv 4 \text{ or }\pm2\mod 8
 	\end{cases}
 	\end{equation}
 	\begin{equation} \label{eq: Jc cond 4}
 	\text{and } t_q\equiv n_q\mod 2.
 	\end{equation}
\end{itemize}

When working with the abbreviated $2$-adic symbol, the Jordan blocks conditions on a compartment of total dimension at least $3$ reduce to just one condition: the total oddity in the compartment has the same parity as its total dimension.

\subsection{Some facts and observations}
We state here some observations. 

For odd $p$, the $p$-excess is always even.

If $L$ is SSF, then since the rank of any invariant factor is square-free, the $p$-adic symbol for $L_p$ will only contains terms for $q=1$ and $q=p$. Furthermore, since the rank of $\Delta(L)$ is at most $\frac{1}{2}\dim(L)$, also $n_p\leq \frac{1}{2}\dim(L)$.

If we take $L^{\mathrm{neg}}$ to be as $L$ with all inner products negated, its local forms will change as follows.
If $p$ is odd, then the $p$-adic symbol for $L^{\mathrm{neg}}_p$ is got from that of $L_p$ by multiplying each superscript by $\left(\frac{(-1)^{n_q}}{p}\right)$. The $2$-adic symbol for $L^{\mathrm{neg}}_p$ is got from that of $L_p$ by negating each subscript.

If $p\nmid \det L$ or if $\left(\frac{-1}{p}\right)=1$, then the $p$-excess of $L_p$ and $L_p^{\mathrm{neg}}$ agree. If $p|L$ and $\left(\frac{-1}{p}\right)=1$, then the $p$-excess of $L_p$ and $L_p^{\mathrm{neg}}$ differ by $4\mod8$.



\section{Lattice embeddings}\label{sec: gluing}

The goal of this section is to prove the following proposition. 

\begin{prop}\label{th: unimodular embedding}
Let $L$ be an integral lattice of signature $+^r-^s$  and let $\delta=\rank(\Delta(L))$. Let $m=\max\{\delta+1,3\}$. Then $L$ embeds in a unimodular lattice of signature $+^{r+m}-^s$.
\end{prop}

The proof will be separated into two cases depending on  the parity of $\det(L)$. The main idea is to use a technique in \cite{Allcock18}. We will construct a lattice $K$ of signature $+^{\delta+1}-^0$ with $\det(K)=(-1)^s\det(L)$ by specifying its local forms $K_p$, chosen such that there exist a group isomorphism $\phi:\Delta(L)\rightarrow\Delta(K)$ which negates norms and inner products. Gluing $L$ to $K$ along the graph of $\phi$ will then result in a unimodular lattice.

\subsection{Case 1: \texorpdfstring{$d$}{d} is odd}

If $\det(L)$ is odd, the following lemma holds regardless of whether $L$ is SSF.

\begin{lemma}\label{lem: d odd}
Let $L$ be an integral lattice of signature $+^r-^s$ with $\det(L)$ odd and $\rank(\Delta(L))=\delta$. Let $m=\max\{\delta+1,3\}$. Then $L$ embeds in a unimodular lattice of signature $+^{r+m}-^s$.
\end{lemma}

\begin{proof}
Assume $L$ is not unimodular and let $d:=(-1)^s\det(L)$. Defined the local forms $K_p$ as follows.

For odd $p\nmid d$, define $K_p$ by $1^{(\frac{d}{p})m}$.

For odd $p|d$ with $d=p^\alpha a$, define $K_p$ by the product of $q^{\epsilon_q n_q}$ where for $q=p^m>1$,  the term $q^{\epsilon_q n_q}$ matches that of $L_p^{neg}$, and where $1^{\epsilon_1 n_1}$ is chosen such that $\prod \epsilon_q =(\frac{a}{p})$ and $\sum n_q=m$.

Let $t\equiv m+\sum_{p\text{ odd}}p\mathrm{-excess}(K_p)$. Since $\sum_{p\text{ odd}}p\mathrm{-excess}(K_p)$ is even, $t$ will have the parity of $m$.
Define $K_2$ by $1^{(\frac{d}{2})m}_t$. 

With these choices of local forms, all conditions \autoref{eq: det condition}-\autoref{eq: Jc cond 4} are satisfied. So there exist an integral lattice $K$ of signature $+^{m}-^0$ and determinant $d$ with the prescribed local forms.

Now each local form $K_p$ differs from $L_p^{neg}$ by a unimodular factor. We have that $\Delta(K_p)$ and $\Delta(L_p^{neg})$ are isomorphic and correspond to the Sylow $p$-subgroups of $\Delta(K)$ and $\Delta(L^{neg})$. It follows that there exist a group isomorphism $\phi:\Delta(L)\rightarrow\Delta(K)$ which negates norms and inner products. Let $G=\{(x,\phi x)\}$ be the graph of $\phi$. Then $G$ is a totally isotropic subgroup of $\Delta(L\oplus K)=\Delta(L)\oplus \Delta(K)$, that is, the natural $\Q/\Z$-valued inner product on $\Delta(L\oplus K)$ vanishes on $G$. 

Write $L\oplus_G K$ for the preimage of $G$ in $(L\oplus K)^*=L^*\oplus K^*$. 
Since $G$ is totally isotropic, $L\oplus_G K$ is an integral lattice containing $L\oplus K$ as a sublattice of index $|G|$ and therefore its determinant is given by
\begin{equation}
\det L\oplus_G K = \frac{\det L\oplus K}{|G|^2} = \frac{d^2}{(d)^2}=1.
\end{equation}
So $L\oplus_G K$ is a unimodular lattice containing $L$ with orthogonal complement $L^\perp$ isomorphic to $K$.
\end{proof}

\subsection{Case 2: \texorpdfstring{$d$}{d} is even}

\begin{lemma}\label{lem: d even}
Let $L$ be an integral lattice of signature $+^r-^s$ with $\det(L)$ even and $\rank(\Delta(L))=\delta$.  Let $m=\max\{\delta+1,3\}$. Then $L$ embeds in a unimodular lattice of signature $+^{r+m}-^s$.
\end{lemma}

\begin{proof}
Let $d:=(-1)^s\det(L)$. Observe first that the SSF assumption guarantees that $\Delta(L_2)$ (the $2$-Sylow subgroup of $\Delta(L)$) is $(\Z/2\Z)^\alpha$ where $d=2^\alpha a$ and $2\nmid a$.
This means that the natural $\Q_2/\Z_2$-valued norms and inner products in $\Delta(L_2)$ are in $\frac{1}{2}\Z_2$. In particular, negating norms and inner products in $\Delta(L_2)$ is trivial.

Defined the local forms $K_p$ as follows.

For odd $p\nmid d$, define $K_p$ by $1^{(\frac{d}{p})m}$.

For odd $p|d$ with $d=p^\alpha a$, define $K_p$ by the product of $q^{\epsilon_q n_q}$ where for $q=p^m>1$,  the term $q^{\epsilon_q n_q}$ matches that of $L_p^{neg}$, and where $1^{\epsilon_1 n_1}$ is chosen such that $\prod \epsilon_q =(\frac{a}{p})$ and $\sum n_q=m$.

Let $t\equiv m+\sum_{p\text{ odd}}p\mathrm{-excess}(K_p)$. Since $\sum_{p\text{ odd}}p\mathrm{-excess}(K_p)$ is even, $t$ will have the parity of $m$.
Define $K_2$ by the reduced $2$-adic symbol $[1^{\left(\frac{a}{2}\right)(m-\alpha)}2^{+ \alpha}]_t$ where $d=2^\alpha a$.

With these choices of local forms, all conditions \autoref{eq: det condition}-\autoref{eq: Jc cond 4} are satisfied. So there exist an integral lattice $K$ of signature $+^{m}-^0$ and determinant $d$ with the prescribed local forms.

Now each local form $K_p$ for $p\neq 2$ differs from $L_p^{neg}$ by a unimodular factor. We then have that for all $p$ (including $p=2$) $\Delta(K_p)$ and $\Delta(L_p^{neg})$ are isomorphic and correspond to the Sylow $p$-subgroups of $\Delta(K)$ and $\Delta(L^{neg})$. It follows that there exist a group isomorphism $\phi:\Delta(L)\rightarrow\Delta(K)$ which negates norms and inner products. Let $G=\{(x,\phi x)\}$ be the graph of $\phi$. Then $G$ is a totally isotropic subgroup of $\Delta(L\oplus K)=\Delta(L)\oplus \Delta(K)$. 

The remaining follows exactly as in the proof of \autoref{lem: d odd}.
\end{proof}



\section{Special subgroups}\label{sec: special}

In this section we gather the necessary ingredients and prove \autoref{th: special} as a direct consequence of \autoref{cor: aut embedding} and \autoref{ert}.

It can be shown that the automorphism group of a non-SSF lattice is always contained in the automorphism group of one which is SSF \cite{Watson1, Watson2,Allcock12}.
For SSF lattices of dimension up to $5$, the rank of their discriminant groups is at most $2$. Therefore, the following corollary follows from \autoref{th: unimodular embedding} since $L\otimes\R$ is the orthogonal complement to the positive subgroups $K\otimes \R$ in the vector space $(L\oplus_G K)\otimes \R$.

\begin{cor}\label{cor: aut embedding}
If $L$ is an integral lattice of signature $+^3-^1$ (resp. $+^4-^1$), then $\Aut(L)$ embeds as a geometrically finite subgroup of $\Aut(\mathbf{I}_{6,1})$ (resp. $\Aut(\mathbf{I}_{7,1})$) which preserves a copy of $\H^3$ (resp. $\H^4$) in $\H^6$ (resp $\H^7$).
\end{cor}

The following theorem of Everitt-Ratcliffe-Tschantz provides the explicit relationship between the automorphism groups of the unimodular lattices $\mathbf{I}_{n,1}$ for $n\leq8$ and right-angled Coxeter groups. Recall that $q_n:=-x_0^2+ x_1^2 + \dots + x_n^2$ and $\Aut(\mathbf{I}_{n,1})=\mathrm{O}(q_n,\Z)$.

\begin{theorem}[{\cite[Theorem 2.1]{ERT}}] \label{ert}
For $2\leq n \leq 7$, $\mathrm{O}^+(q_n,\Z)_{(2)}$ is a geometric RACG. It is the reflection group of an all-right hyperbolic polyhedron of dimension $n$.
The group $\mathrm{O}^+(q_8,\Z)_{(2)}$ contains a geometric RACG as a subgroup of index 2. This subgroup is the reflection group of an all-right hyperbolic polyhedron of dimension $8$.
\label{th: RACG}
\end{theorem}

\begin{lemma}\label{lem: cong2 embeds}
Let $L$ be an integral lattice and $K$ the corresponding integral lattice with $L\oplus_G K$ unimodular as constructed in the proof of \autoref{th: unimodular embedding}. Then $\Aut(L)_{(2)}$ is contained in $\Aut( L\oplus_G K)_{(2)}$. 
\end{lemma}

\begin{proof}
Recall that $K$ was constructed so that $\Delta(L)\cong\Delta(K)\cong G$. An element $\gamma\in\Aut(L)$ is in $\Aut(L)_{(2)}$ if and only if for any $v\in L$, $v\gamma=v+2v'$ where $v'$ is some other element in $L$, or equivalently, $v\gamma\equiv v\mod2$. Similarly, $\gamma\in\Aut(L\oplus_G K)$ if and only if for any $v\in L\oplus_G K$, $v\gamma\equiv v\mod2$. If $v\in L\oplus K$ then clearly for $\gamma\in Aut(L)$, $v\gamma\equiv v\mod2$.

For the remainder of the proof, set $\gamma\in Aut(L)_{(2)}$ and $u\in L\oplus_G K - L\oplus K$. Suppose $m$ is the smallest positive integer so that $mu=u'\in L\oplus K$.
Then $u'\gamma=u'+2x$ where $x\in L\oplus K$.
If $v$ is some other element of $L\oplus K$ with $u'\perp v$, then $v\gamma=v+2y$ for some $y\in L\oplus K$ and since $0=(u',v)=(u'\gamma,v\gamma)$,
\begin{equation*}
 0=(u'\gamma,v+2y)=(u'\gamma,v)+(mu\gamma,2y)=(u'\gamma,v)+2m(u\gamma,y).
\end{equation*}
Thus $(u'\gamma,v)\equiv 0\mod 2m$. Also,
\begin{equation*}
(u'\gamma,u')=(u'+2x,u')=(u',u')+(2x,mu)=(u',u')+2m(x,u).
\end{equation*}
Thus $(u'\gamma,u')\equiv (u',u')\mod 2m$. Therefore $u'\gamma\equiv u'\mod 2m$ and since $u'=mu$ we have $u\gamma\equiv u\mod 2$. This shows $\gamma\in\Aut(L\oplus_G K)_{(2)}$.
\end{proof}

\begin{proof}[Proof of \autoref{th: special}]
The proof follows from \autoref{cor: aut embedding}, \autoref{th: RACG}, and \autoref{lem: cong2 embeds} since the lattice $L$ embeds as the orthogonal complement of the positive definite sublattice $K$ and this induces the embedding of $\Aut(L)$ as a geometrically finite subgroup of $\Aut(L\oplus_G K)$.
\end{proof}

We remark here that the principal congruence subgroups of level 2 have uniformly bounded index in integral arithmetic subgroups of simplest type in $\Isom(H^n)$. Indeed, any such $\Gamma$ is contained in some $\mathrm{O}^+(f;\Z)\subset \GL_{n+1}(\Z)$ and thus $[\Gamma:\Gamma_{(2)}]\leq |\GL_{n+1}(\Z/2\Z)|$.

The Bianchi group $\PSL_2(\mathcal{O}_d)$ is contained in $\mathrm{O}^+(f_d;\Z)$ where
$$ f_d = \begin{cases}
2x_0x_1+2x_2^2+2dx_3^2 & \text{if } d\equiv 1,2\mod 4\\
2x_0x_1+2x_2^2+2x_2x_3+\frac{d+1}{2} x_3^2 & \text{if } d\equiv 3\mod 4 
\end{cases} $$
\cite{JM}. It is easy to see via the explicit embedding of $\PSL_2(\mathcal{O}_d)$ into $\mathrm{O}^+(f_d;\Z)$ described in \cite[\S 3.1]{bianchi} that the image of $\PSL_2(\mathcal{O}_d)_{(2)}$ lands in $\mathrm{O}^+(f_d;\Z)_{(2)}$. Therefore, we get that the principal congruence subgroup $\PSL(2,\mathcal{O}_d)_{(2)}$ embeds in a RACG. The index of this subgroup is well known to be \begin{equation*}
[\PSL(2,\mathcal{O}_d): \PSL(2,\mathcal{O}_d)_{(2)}]= \begin{cases}
	48 & \text{if }d\equiv 1,2 \mod (4) \\
	60 & \text{if }d\equiv 3 \mod (8) \\
	36 & \text{if }d\equiv 7 \mod (8),
\end{cases}
\end{equation*}
(see for example \cite{Dickson}), and thus, \autoref{prop: bianchi} follows.



\section{An example}\label{sec: example}

Consider the lattice $L$ with basis $\{e_0,e_1,e_2,e_3\}$ and associated inner product matrix $ A_L=\diag\{-7,1,1,1\}$.
The dual $L^*$ has $\Z$-basis $\{\frac{1}{7}e_0,e_1,e_2,e_3\}$ and the discriminant group $\Delta(L)\cong\Z/7\Z$ is generated  by the image of the vector $\frac{1}{7}e_0$. 
We can take $K$ to be the lattice with basis $\{e_4,e_5,e_6\}$ and associated inner product matrix$A_K=\diag\{7,1,1\}$.
The dual $K^*$ has $\Z$-basis $\{\frac{1}{7}e_4,e_5,e_6\}$ and the discriminant group $\Delta(K)\cong\Z/7\Z$ is generated  by the image of the vector $\frac{1}{7}e_4$. Therefore $\Delta(L\oplus K)=\Z/7\Z\times \Z/7\Z$ and image of the vector $u=\frac{4}{7}e_0+\frac{3}{7}e_4$ generates a totally isotropic subgroup $G$ of $\Delta(L\oplus K)$ of order $\Z/7\Z$. Indeed, $u$ has norm $-1$ and pairs integrally with $L\oplus K$. If we let $v=\frac{3}{7}e_0+\frac{4}{7}e_4$ then $\{u,v,e_1,e_2,e_3,e_5,e_6\}$ defines an orthogonal $\Z$-basis for $L\oplus_G K$ with associated inner product matrix $\diag\{-1,1,1,1,1,1,1\}$. Therefore $L\oplus_G K = \mathbf{I}_{6,1}$.

Since $e_0=4u-3v$, a copy of $L$ is the sublattice with $\Z$-basis $\{4u-3v,e_1,e_2,e_3\}$ which is the orthogonal complement of a copy of $K$ with $\Z$-basis $\{-3u+4v,e_5,e_6\}$. The change of basis matrix is
\begin{equation} B= \begin{pmatrix}
\frac{4}{7}& 0 & \frac{-3}{7} & \\ 0 & I_3 & 0 & \\
\frac{-3}{7} & 0 & \frac{4}{7} & \\ & & & I_2
\end{pmatrix} \end{equation}
where $I_n$ denotes the $n\times n$ identity matrix. Therefore, $$B\cdot\diag\{-7,1,1,1,7,1,1\}\cdot B^{tr}=\diag\{-1,1,1,1,1,1,1\}.$$

Let $\gamma$ be in $\Aut(L)$. As a matrix with entries in $\Z$ preserving $A_L$,  $\gamma$ the first row $\begin{pmatrix} a&b&c&d \end{pmatrix}$ satisfies $a^2\equiv1\mod 7$ and $b,c,d\equiv1\mod7$. We extend $\gamma$ in $\Aut_L$ to $\gamma$ in $\Aut(L\oplus K)$ as follows:
\begin{equation}
\gamma \mapsto \begin{cases}
\begin{pmatrix} \gamma & \\ & I_3 \end{pmatrix} & \text{if }a\equiv 1\mod7 \\
\begin{pmatrix} \gamma & & \\ & -I_2 &  \\ & & I_1 \end{pmatrix} & \text{if }a\equiv -1\mod7
\end{cases}
\end{equation}
and such an integral matrix preserves integrality when conjugated by $B$, that is, $B\cdot\gamma\cdot B^{-1}$ is integral and preserves $\diag\{-1,1,1,1,1,1,1\}$. Indeed, if
$$ \gamma_1 = \begin{pmatrix}
7 a+1 & 7 b & 7 c & 7 d & \\
 e & f & g & h & \\
 i & j & k & l & \\
 m & n & o & p & \\
 & & & & I_3 
\end{pmatrix}
\text{, } \gamma_2 = \begin{pmatrix}
7 a-1 & 7 b & 7 c & 7 d & & \\
 e & f & g & h & & \\
 i & j & k & l & & \\
 m & n & o & p & & \\
 & & & & -I_2 & \\
 & & & & & I_1
\end{pmatrix} $$
then 
$$ B\cdot\gamma_1\cdot B^{-1} = \begin{pmatrix}
16 a+1 & 4 b & 4 c & 4 d & 12 a & 0 & 0 \\
 4 e & f & g & h & 3 e & 0 & 0 \\
 4 i & j & k & l & 3 i & 0 & 0 \\
 4 m & n & o & p & 3 m & 0 & 0 \\
 -12 a & -3 b & -3 c & -3 d & 1-9 a & 0 & 0 \\
 0 & 0 & 0 & 0 & 0 & 1 & 0 \\
 0 & 0 & 0 & 0 & 0 & 0 & 1 
\end{pmatrix} $$
and
$$ B\cdot\gamma_2\cdot B^{-1} = \begin{pmatrix}
16 a-1 & 4 b & 4 c & 4 d & 12 a & 0 & 0 \\
 4 e & f & g & h & 3 e & 0 & 0 \\
 4 i & j & k & l & 3 i & 0 & 0 \\
 4 m & n & o & p & 3 m & 0 & 0 \\
 -12 a & -3 b & -3 c & -3 d & -9 a-1 & 0 & 0 \\
 0 & 0 & 0 & 0 & 0 & -1 & 0 \\
 0 & 0 & 0 & 0 & 0 & 0 & 1 
\end{pmatrix} . $$

Furthermore, If $\gamma\equiv I_7\mod2$ entrywise, so is $B\cdot\gamma\cdot B^{-1}$, and thus there is an inclusion of principal congruence subgroups of level $2$: $\Aut(L)_{(2)}\subset\Aut(\mathbf{I}_{6,1})_{(2)}$.



\bibliographystyle{alpha}
\bibliography{bib}

\end{document}